\documentclass[leqno, 11pt]{amsart}

\usepackage{amssymb, amsmath}
\usepackage{amsthm, amsfonts, mathrsfs}
\def\inte#1{
\displaystyle\mathop{#1\kern0pt}^\circ }

\def\virgp{\raise 2pt\hbox{,}}
\def\cdotpv{\raise 2pt\hbox{;}}

\def\C{\mathop{\mathbb C\kern 0pt}\nolimits}
\def\DD{\mathop{\mathbb D\kern 0pt}\nolimits}
\def\EE{\mathop{{\mathbb E \kern 0pt}}\nolimits}
\def\K{\mathop{\mathbb K\kern 0pt}\nolimits}
\def\N{\mathop{\mathbb N\kern 0pt}\nolimits}
\def\Q{\mathop{\mathbb Q\kern 0pt}\nolimits}
\def\R{\mathop{\mathbb R\kern 0pt}\nolimits}
\def\SS{\mathop{\mathbb S\kern 0pt}\nolimits}
\def\ZZ{\mathop{\mathbb Z\kern 0pt}\nolimits}
\def\TT{\mathop{\mathbb T\kern 0pt}\nolimits}
\def\P{\mathop{\mathbb P\kern 0pt}\nolimits}

\newcommand{\beq}{\begin{equation}}
\newcommand{\eeq}{\end{equation}}
\newcommand{\ben}{\begin{eqnarray}}
\newcommand{\een}{\end{eqnarray}}
\newcommand{\beno}{\begin{eqnarray*}}
\newcommand{\eeno}{\end{eqnarray*}}

\newtheorem{defi}{Definition}[section]
\newtheorem{thm}{Theorem}[section]
\newtheorem{lem}{Lemma}[section]
\newtheorem{rmk}{Remark}[section]

\newtheorem{prop}{Proposition}[section]
\renewcommand{\theequation}{\thesection.\arabic{equation}}

\allowdisplaybreaks[4]
\begin{document}

\title[]
{Global regularity for the 3D Hall-MHD equations with low regularity axisymmetric data}

\author[]{Zhouyu Li}
\address[Z. Li]{Department of Applied Mathematics, Northwestern Polytechnical University, Xi'an 710129, China}
\email{zylimath@163.com}

\author[]{Pan Liu}
\address[P. Liu]{School of Mathematics and Statistics, Beijing Institute of Technology, Beijing 100081, China}
\email{3120185715@bit.edu.cn}

\begin{abstract}
In this paper, we consider the global well-posedness of the incompressible Hall-MHD equations in $\mathbb{R}^3$. We prove that the solution of this system is globally regular if the initial data
is axisymmetric and the swirl components of the velocity and magnetic vorticity are trivial.
It should be pointed out that the initial data without any smallness and  in low regularity spaces.
This improves a previous result established in \cite{Fan2013}.
\end{abstract}


\date{}

\maketitle



\vskip 0.2cm

\renewcommand{\theequation}{\thesection.\arabic{equation}}
\setcounter{equation}{0}
\section{Introduction}
This paper is concerned with the following 3D incompressible Hall-MHD equations
\begin{equation}\label{Hall-MHD}
    \begin{cases}
    \partial_t u + u\cdot \nabla u + \nabla P - \Delta u  = B\cdot \nabla B,\quad (t, x)\in \mathbb{R}^+\times \mathbb{R}^3,\\
    \partial_t B + u\cdot \nabla B + \operatorname{curl}(\operatorname{curl}B \times B)
         = \Delta B+ B\cdot \nabla u ,\\
    \operatorname{div} u = \operatorname{div} B = 0,
    \end{cases}
\end{equation}
with the initial data $\left( u(x, 0), B(x, 0)\right)=\left( u_0(x), B_0(x)\right)$.
Here $u$ and $B$ denote the fluid velocity field and the magnetic field respectively; the scalar function $P$ denotes the pressure.

The Hall term $\operatorname{curl}(\operatorname{curl}B \times B)$ is derive from Ohm's law and
describe deviation from charge neutrality between the electrons and the ions. Hence, the Hall-MHD equations can describe some physical phenomena that can not be characterized appropriately by the classical MHD equations, for example, magnetic
reconnection in plasmas, neutron stars, star formation, etc. For more physical background, see \cite{H2005, L1960}.

Recently, there is been tremendous interest in developing the Hall-MHD system. In \cite{A2011}, Acheritogaray et al. gave
a derivation of the Hall-MHD system from either two-fluid or kinetic models in a mathematically rigorous way.
Chae et al. in \cite{Chae2014} established the global existence of weak solutions and the local well-posedness
of smooth solutions in Sobolev space $H^s(\mathbb{R}^3)$ with $s>\frac{5}{2}$. They also proved a Liouville theorem for the stationary solutions of the system \eqref{Hall-MHD}. In \cite{Ben2016}, the authors showed the local existence of strong solutions in $H^2(\mathbb{R}^3)$. Later on, Dai \cite{Dai2019} weakened the result in \cite{Ben2016}.
More precisely, they obtained the local well-posedness theory in $H^s(\mathbb{R}^n)$ with $s>\frac{n}{2}$ and $ n\geqslant 2$.
Very recently, Danchin et al. in \cite{Dan2020} showed global well-posedness when the initial conditions in critical spaces
$\dot{B}^{\frac{3}{p}-1}_{p,1}$, $ 1 \leqslant p<\infty$, and are small enough. More interesting results, we recommend
\cite{LY2020, Li2020, Liu2021, Wu2018}.

It is well-known that the global well-posedness to the 3D incompressible Hall-MHD equations with large initial data is still unsolved. As far as we know, with respect to this
matter, only some partial results are known. For example, the solutions satisfies some special structures where an important
case is axisymmetric. Fan et al. in \cite{Fan2013} investigated \eqref{Hall-MHD} is global well-posed without any smallness
assumptions for a class of special axisymmetric initial data. More precisely, the initial data satisfy swirl components of
the velocity field and magnetic vorticity field vanish, that is,
\begin{equation*}
     \begin{split}
u(0, x)=u_0^r(r, z)e_r + u_0^z(r, z)e_z \quad \mbox{and} \quad B(0, x)=B_0^\theta(r, z)e_\theta.
     \end{split}
\end{equation*}
And they also assumed that
\begin{equation*}\label{assume}
     \begin{split}
(u_0, B_0)\in H^2(\mathbb{R}^3), \quad \mbox{and} \quad \frac{B_0^\theta}{r}\in L^\infty(\mathbb{R}^3).
     \end{split}
\end{equation*}

Inspired by \cite{Fan2013}, the aim of this paper is to establish the global well-posedness of 3D Hall-MHD equations
\eqref{Hall-MHD} for a class of large axisymmetric data without swirl, which is stated as follows.

\begin{thm}\label{thm-main}
Let $u_0$ and $B_0$ are both axisymmetric divergence free vector fields such that $u_0^\theta=B_0^r=B_0^z=0$.
Assume that $(u_0, \,B_0)\in H^{1}{(\mathbb{R}^3)}\times H^2{(\mathbb{R}^3)}$,
and $ (\omega_{0}, \,\frac{\omega_0}{r}, \, \frac{B_{0}}{r})
\in L^{\sigma}(\mathbb{R}^3) \times L^2(\mathbb{R}^3) \times L^{\alpha}(\mathbb{R}^{3}) $ with $ 6< \alpha \leq \infty $ and $ 3 < \sigma \leqslant \alpha $.  Then the system \eqref{Hall-MHD} has a unique
global solution $(u, B)$ satisfying
\begin{align*}
&u\in L^\infty(0, T; H^1(\mathbb{R}^3))\cap L^1(0, T; H^3(\mathbb{R}^3))
  \cap L^1(0, T; W^{1, \infty}(\mathbb{R}^3)),\\
& B\in L^\infty(0, T; H^2(\mathbb{R}^3)) \cap L^{2}(0,\, T; \,H^{3}(\mathbb{R}^{3}))
  \cap L^1(0, T; W^{1, \infty}(\mathbb{R}^3)), \\
& \omega \in  L^\infty(0, T; L^\sigma(\mathbb{R}^3)) \cap
    L^{2}(0, T; H^1(\mathbb{R}^3)) \cap
   L^{1}(0, T; H^2(\mathbb{R}^3)), \\
&  \frac{\omega}{r} \in  L^\infty(0, T; L^2(\mathbb{R}^3)) \cap
    L^{2}(0, T; H^1(\mathbb{R}^3)).
\end{align*}
for any $0< T< \infty$.
\end{thm}

\begin{rmk}
Using the estimate $u\in L^1(0, T; W^{1, \infty}(\mathbb{R}^3))$, which is given by Theorem \ref{thm-main}, we can propagate by classical arguments higher order regularity, for example higher $H^s$ Sobolev regularity.
\end{rmk}

\begin{rmk}
It should be stressed that we improve the result of \cite{Fan2013} in two ways.
On the one hand, since $B(0, x)=B_0^\theta(r, z)e_\theta$, we extend the condition $\frac{B_0^\theta}{r}\in L^\infty(\mathbb{R}^3)$ to $\frac{B_0^\theta}{r}\in L^\alpha(\mathbb{R}^3)$, $6< \alpha\leqslant \infty$.
On the other hand, due to in cylindrical coordinates the vorticity of the swirl-free axisymmetric velocity
is given by
$\omega= \operatorname{curl}u = \omega^\theta e_\theta$
with $\omega^\theta :=\partial_z u^r-\partial_r u^z$, and
$|\nabla^2 u| \sim |\nabla \omega^\theta| + |\frac{\omega^\theta}{r}|.$ Moreover, Applying the Sobolev embedding inequality, we know $\|\nabla u\|_{L^p(\mathbb{R}^3)}\leqslant C\|\nabla u\|_{H^1(\mathbb{R}^3)}$, $2\leqslant p\leqslant 6$.
Thus, our result weaken the condition $u_0\in H^2(\mathbb{R}^3)$.
\end{rmk}

\begin{rmk}\label{rmk1-3}
Due to the presence of the Hall term $\operatorname{curl}(\operatorname{curl}B \times B)$, we need the
estimate $\nabla B \in L^{2}(0, T; L^{\infty}(\mathbb{R}^{3}))$ to obtain the $L^{\infty}(0, T; H^{2}(\mathbb{R}^3))$ estimate of $B$, for more details $J_4$ in step 8 below. Thus we add the assumption condition
$(\omega_{0}, \, \frac{B_{0}}{r})\in L^{\sigma}(\mathbb{R}^3) \times L^{\alpha}(\mathbb{R}^{3}) $ with $6< \alpha\leq \infty $ and $ 3 < \sigma \leqslant \alpha $. In the short future we will further study whether or not the condition can be removed.
\end{rmk}

\begin{rmk}
When the Hall term $\operatorname{curl}(\operatorname{curl}B \times B)$ is neglected, the system \eqref{Hall-MHD} reduces to the
incompressible viscous MHD system. For this case,  combining remark \ref{rmk1-3}, we do not need the assumption condition $(\omega_{0}, \, \frac{B_{0}}{r})\in L^{\sigma}(\mathbb{R}^3) \times L^{\alpha}(\mathbb{R}^{3}) $ with $ 6< \alpha\leqslant \infty $ and $ 3 < \sigma \leqslant \alpha $. This is because we emphasize that for the incompressible viscous MHD system, in order to get the estimate $B\in L^{\infty}(0, T; H^{2}(\mathbb{R}^3))$, we only need the $L^{2}(0, T; L^6(\mathbb{R}^{3}))$ estimate of $\nabla B$, which can be similarly estimated as \eqref{H2Est-B} below. This is a new result for incompressible viscous MHD system.
\end{rmk}

The proof of the main result is achieved by using more deeply the structure of the Hall-MHD equations in axisymmtric case with zero swirl component of the velocity and the magnetic voricity. 
In contrast with the proof in \cite{Fan2013}, since the absence of the condition $u_0\in H^2(\mathbb{R}^3)$, we use the dyadic decomposition to obtain the $u\in L^1([0, T]; \operatorname{Lip} (\mathbb{R}^3))$, see step 5 below. We also can not directly obtain $\nabla B\in L^2([0, T]; L^\infty (\mathbb{R}^3))$, which plays the key role in the proof. For that we assume $(\omega_{0}, \, \frac{B_{0}}{r})\in L^{\sigma}(\mathbb{R}^3) \times L^{\alpha}(\mathbb{R}^{3}) $ with $ 6 < \alpha \leqslant \infty $ and $ 3 < \sigma \leqslant \alpha $. It should be emphasized that we sufficiently use the embedding theorem $W^{1, p}(\mathbb{R}^3)\hookrightarrow C_0(\mathbb{R}^3) $ $ (p > 3) $ 
(see Theorem 6.7 of \cite{LT2007}) for appropriately choose the value range of $\sigma$ and $\alpha$, for more details see Step 4 and Step 6 below. With $u\in L^1([0, T]; \operatorname{Lip} (\mathbb{R}^3))$ and $\nabla B\in L^2([0, T]; L^\infty (\mathbb{R}^3))$ in hand, we show the $H^2$ estimate of $B$ in Step 8.

\medbreak \noindent{\bf Notations:} We shall denote
$\int \cdot \mathrm{d}x \triangleq \int_{\mathbb{R}^3} \cdot \mathrm{d}x $,
$ \| \cdot \|_{L^{p}} \triangleq \| \cdot \|_{L^{p}(\mathbb{R}^{3})} $, and $ \| (A, B, C) \|_{X}^{r} = \| A \|_{X}^{r}
+ \| B \|_{X}^{r} + \| C \|_{X}^{r} $, where X is a Banach space. We use the letter $C$ to denote a generic constant, which may vary from line to line. We always use $A\lesssim B$ to denote $A \leqslant CB$ and omit a generic positive constant $C$ in $\exp C t^\alpha$.


\renewcommand{\theequation}{\thesection.\arabic{equation}}
\setcounter{equation}{0} 

\section{Preliminaries}
In this section, we will transform the system \eqref{Hall-MHD} into the cylindrical coordinate, recall some useful inequalities and introduce Besov spaces.

Firstly, we derive the system \eqref{Hall-MHD} in the cylindrical coordinate.
In the cylindrical coordinate, the solutions to \eqref{Hall-MHD} have the following form
\begin{equation*}
\begin{cases}
\, u(t, x)=u^r(r, z, t) e^r + u^\theta(r, z, t) e^\theta
      + u^z(r, z, t) e^z,\\
\, B(t, x)=B^r(r, z, t) e^r + B^\theta(r, z, t) e^\theta
      + B^z(r, z, t) e^z,\\
\, P(t,x)= P(r, z, t).
\end{cases}
\end{equation*}
where
\[
   e^{r} = \left(\frac{x_{1}}{r}, \,\frac{x_{2}}{r}, \,0\right)^{T},  \
   e^{\theta} = \left(-\frac{x_{2}}{r}, \,\frac{x_{1}}{r}, \,0\right)^{T}, \
   e^{z} = (0, \,0, \,1)^{T}, \
    r = \sqrt{x_{1}^{2} + x_{2}^{2}}.
\]

Taking advantage of the local well-posedness result for the
system \eqref{Hall-MHD} in $\mathbb{R}^3$, see \cite{Dai2019},
we can obtain the following lemma.
\begin{lem}\label{lem-0}
Let $(u_0, B_0) \in H^{1}{(\mathbb{R}^3)} \times H^2{(\mathbb{R}^3)}$ be axisymmetric divergence-free vector.
Then there exists $T>0$ and a unique axisymmetric
solution $(u, \,B)$ on $[0, \,T)$ to the
system \eqref{Hall-MHD} such that
\begin{equation*}
 u \in L^\infty(0,\, T; \;H^1(\mathbb{R}^3)),\qquad
 B \in L^\infty(0,\, T; \;H^2(\mathbb{R}^3)).
\end{equation*}
\end{lem}

As in \cite{Fan2013}, we also assume  $ u_0^\theta = B_0^r = B_0^z=0 $ and seek the axisymmetric solutions to the system \eqref{Hall-MHD}.
It is not difficult to deduce from the uniqueness of local solutions that
$ u_0^\theta=B_0^r=B_0^z=0 $ implies $ u^\theta = B^r = B^z = 0 $ for all later times.
In this case, by directly computations, we have
\begin{equation}   \label{doubleCurl}
    \operatorname{curl}(\operatorname{curl}B \times B)
  = -\frac{2}{r} B^{\theta} \partial_{z}B^{\theta}  e^{\theta},  \qquad
    \omega =  \operatorname{curl}u
           = (\partial_{z}u^{r} - \partial_{r}u^{z}) e^{\theta}.
\end{equation}
For convenience, we denote $  \omega^{\theta} = \partial_{z}u^{r} - \partial_{r}u^{z} $, i.e.,
$ \operatorname{curl}u = \omega^{\theta} e^{\theta} $.

Hence the system \eqref{Hall-MHD} can be rewritten as
\begin{equation}\label{MHD-2}
    \begin{cases}
   \partial_t u^r + u^r\partial_r u^r + u^z \partial_z u^r + \partial_r P
  = (\Delta-\frac{1}{r^2})u^r -\frac{(B^\theta)^2}{r},\\
   \partial_t u^z + u^r \partial_r u^z + u^z \partial_z u^z + \partial_z P = \Delta u^z, \\
   \partial_t B^\theta + u^r\partial_r B^\theta + u^z\partial_z B^\theta
  = (\Delta - \frac{1}{r^{2}}) B^{\theta} + \frac{u^r }{r}B^{\theta}
    + \frac{1}{r}\partial_{z}(B^\theta)^{2},\\
   \partial_r u^r+\frac{u^r}{r}+\partial_z u^z =0,\\
  u^{r}\big|_{t = 0} = u^{r}_{0}, \; u^{z}\big|_{t = 0} = u^{z}_{0}, \;
  B^{\theta}\big|_{t = 0} = B^{\theta}_{0}.
    \end{cases}
\end{equation}

Define
\begin{equation}\label{1}
     \begin{split}
\Pi =\frac{B^\theta}{r} \quad \mbox{and} \quad \Omega=\frac{\omega^\theta}{r},
     \end{split}
\end{equation}
then the system \eqref{MHD-2} is equivalent to
\begin{equation}\label{Hall-MHD-3}
    \begin{cases}
    \partial_t \Pi + u \cdot \nabla \Pi
  = (\Delta + \frac{2}{r}\partial_{r} ) \Pi + 2 \Pi \partial_{z} \Pi,\\
    \partial_t \Omega + u\cdot \nabla \Omega
    = (\Delta + \frac{2}{r}\partial_r)\Omega -\partial_z \Pi^2.
    \end{cases}
\end{equation}

Secondly, we present two known lemma, which will play an important role in our proof.
\begin{lem}[\cite{Abidi2011, Lei2015}] \label{lem-BS}
Let $ u $ be a smooth  axisymmetric vector field with zero divergence.
Then
\begin{equation*}
     \begin{split}
   \|u\|_{L^\infty}\lesssim \| \omega \|_{L^2}^{\frac{1}{2}}\| \nabla\omega \|_{L^2}^{\frac{1}{2}},
   \qquad
   \|\frac{u^r}{r}\|_{L^\infty}\lesssim \|\Omega\|_{L^2}^{\frac{1}{2}}\|\nabla\Omega\|_{L^2}^{\frac{1}{2}}.
     \end{split}
\end{equation*}
\end{lem}
This lemma was proved in many literatures, such as \cite{Abidi2011} and \cite{Lei2015}, we omit the details here.

\begin{lem}[Lemma 4.1 of \cite{Liu2018}]   \label{Lem:grad-curl}
Let $ u $ be divergence-free and $ 1 < p < \infty $. Then
there is a constant $ C > 0 $ depending only on the dimension $ n $ such that
\begin{equation*}
    \| \nabla u \|_{L^{p}} \leqslant \frac{C p^{2}}{p - 1}
    \| \operatorname{curl}u \|_{L^{p}}.
\end{equation*}
\end{lem}

Thirdly, we review the definition of Besov spaces and some useful inequalities. Let us first recall the classical dyadic decomposition in $ \mathbb{R}^3 $, see \cite{Che1998}.

For every $ u \in \mathcal{S}'(\mathbb{R}^3) $, the class of tempered distributions,
and $ q \in \mathbb{N} $
\begin{equation*}
    \begin{split}
   \Delta_q u:= \varphi (2^{-q}\mathcal{D})u, \quad \Delta_{-1} u:=\chi(\mathcal{D})u
   \quad \mbox{and} \quad S_q u:= \sum_{-1\leqslant j \leqslant q-1} \Delta_j u.
    \end{split}
\end{equation*}
where $ \varphi $ and $ \chi $ are two smooth functions with compact support satisfying
\begin{equation}   \label{Defannulus}
\begin{split}
      &\operatorname{supp} \varphi = \mathcal{C}
 := \left\{ \xi \in \mathbb{R}^3 \,\mid\, \frac{3}{4}\leqslant |\xi|\leqslant \frac{8}{3} \right\},  \\
     &\operatorname{supp} \chi = \mathcal{B}
 :=\left\{ \xi \in \mathbb{R}^3 \,\mid\, |\xi| \leqslant \frac{3}{4} \right\},
\end{split}
\end{equation}
\begin{equation*}
   \sum_{j\in \mathbb{Z}}\varphi (2^{-j}\xi) = 1
   \quad \forall\, \xi \in \mathbb{R}^{3} \!\setminus\! \{ 0 \},
\quad \textrm{and}  \quad
  \chi(\xi)+\sum_{q\geqslant 0}\varphi (2^{-q}\xi)= 1   \quad \forall\,\xi\in \mathbb{R}^3.
\end{equation*}
Then for every tempered distribution $ u $, the following decomposition holds
\begin{equation*}
    \begin{split}
   u=\sum_{q \geqslant -1}\Delta_q u,\quad \forall \, u \in \mathcal{S}'(\mathbb{R}^3).
    \end{split}
\end{equation*}

Now we are in a position to give a definition to the
Besov space $ B_{p, r}^{s}(\mathbb{R}^{3}) $, abbreviated as $ B_{p, r}^{s} $ without confusion.

\begin{defi}[\cite{Che1998}]
For $ 1 \leqslant p , \,r \leqslant \infty $ and $ s \in \mathbb{R}$,
the Besov space $ B_{p, r}^s $ is the space of tempered distribution $ u $ such that
\begin{equation*}
   \|u\|_{B^s_{p, r}}:= \left(2^{qs} \|\Delta_q u\|_{L^p}\right)_{\ell^r}< +\infty.
\end{equation*}
\end{defi}

Finally, we recall the Bernstein inequalities, 
tame estimates, and commutator estimates.
\begin{lem}[Lemma 2.1 of \cite{Che1998}]\label{lem-Bern}
Let $\mathcal{B}$ be a ball and $\mathcal{C}$ an annulus in $ \mathbb{R}^3 $,
defined in \eqref{Defannulus}. There exists a constant $C$ such
that for any positive $ \delta $, non-negative integer $k$,
smooth homogeneous function $ \sigma $ of degree
$ m $, real numbers $ q \geqslant p \geqslant 1 $, and $ u \in L^{p}(\mathbb{R}^{3}) $,
we have
\begin{equation*}
    \begin{split}
   &\operatorname{supp} \hat{u}\subset \delta \mathcal{B} \Rightarrow \sup_{|\alpha|=k}\|\partial^\alpha u\|_{L^q}
   \leqslant C^{k+1} \delta^{k+3(\frac{1}{a}-\frac{1}{b})} \|u\|_{L^p}, \\
   &\operatorname{supp} \hat{u}\subset \delta \mathcal{C} \Rightarrow C^{-1-k} \delta^{k}\|u\|_{L^p}
   \leqslant \sup_{|\alpha|=k} \|\partial^\alpha u\|_{L^p}\leqslant C^{1+k} \delta^{k}\|u\|_{L^p}, \\
   &\operatorname{supp} \hat{u}\subset \delta \mathcal{C} \Rightarrow \
   \|\sigma(\mathcal{D}) u\|_{L^q}\leqslant C_{\sigma, m}\delta^{m+3(\frac{1}{a}-\frac{1}{b})}\|u\|_{L^p},
    \end{split}
\end{equation*}
where $ \hat{u} $ denotes the Fourier transform of $ u $.
\end{lem}

It is not hard to see that $ H^s = B^s_{2, 2} $. We deduce from the Bernstein
inequalities that the following continuous embedding holds
\begin{equation*}
  B^s_{p_1, r_1}\hookrightarrow B^{s+ 3(\frac{1}{p_2}-\frac{1}{p_1})}_{p_1, r_1}
\end{equation*}
for $ p_1 \leqslant p_2 $ and $ r_1 \leqslant r_2 $.

The following inequality is called tame estimates, see \cite{Che1998}.
\begin{lem}[\cite{Che1998}]\label{lem-Besov}
Suppose that $s< \frac{3}{p}$ or $s=\frac{3}{p}$, $q=1$. Then there exists a constant $C>0$ such
that
\begin{equation*}
    \begin{split}
   \|fg\|_{B^s_{p, q}(\mathbb{R}^3)}
   \leqslant \frac{C^{s+1}}{s}\left( \|f\|_{L^{\infty}(\mathbb{R}^3)}\|g\|_{B^s_{p, q}(\mathbb{R}^3)} + \|g\|_{L^{\infty}(\mathbb{R}^3)}\|f\|_{B^s_{p, q}(\mathbb{R}^3)} \right),
    \end{split}
\end{equation*}
\end{lem}

\begin{lem}[Commutator Estimates]  \label{lem:Commun}
Let $ 1 < p < \infty $, $ s > 0 $ and $ \Lambda^s:= (-\Delta)^{\frac{s}{2}} $.
Then for any $ f, \,g \in C_{c}^{\infty}(\mathbb{R}^{3}) $,

\noindent
{\rm (1) (Lemma 3.1 of \cite{Ju2004}) }
\begin{equation*}
  \|\Lambda^s (fg)- f\Lambda^s g\|_{L^p}
  \leqslant C\left(\|\nabla f\|_{L^{p_1}}\|\Lambda^{s-1} g\|_{L^{q_1}} +
   \|\Lambda^s f\|_{L^{p_2}}\|g\|_{L^{q_2}}\right),
\end{equation*}
where  $ 1 < p_1, p_2, q_1, q_2 < \infty $ satisfy
\[
\frac{1}{p}=\frac{1}{p_1}+\frac{1}{q_1}=\frac{1}{p_2}+\frac{1}{q_2}.
\]
{\rm (2) (Lemma X1 of \cite{Kato1988}) }
\begin{equation*}
  \| \Lambda^s (fg)- f \Lambda^s g \|_{L^p}
  \leqslant C \left( \| \nabla f \|_{L^{\infty}} \| \Lambda^{s-1} g \|_{L^{p}} +
   \| \Lambda^s f \|_{L^{p}} \| g \|_{L^{\infty}} \right).
\end{equation*}
\end{lem}


\renewcommand{\theequation}{\thesection.\arabic{equation}}
\setcounter{equation}{0} 

\section{Proof of Theorem \ref{thm-main}}
The purpose of this section is to derive some priori estimates and then complete the proof of Theorem \ref{thm-main}.
Let us first state the priori estimates for the system \eqref{Hall-MHD}.
\begin{prop} \label{prop-1}
Under the assumptions of Theorem \ref{thm-main}, if $ (u, \,B) $ is a smooth solution
to the system \eqref{Hall-MHD} then there is a constant $ C $ depending on $ T $ and $ (u_{0}, \,B_{0}) $ such that
\begin{equation*}
    \| (u(t), \, B(t)) \|_{H^2}^2
  + \int_{0}^{t}{\| u(s) \|_{H^3}}\mathrm{d}s
  + \int_{0}^{t}{\| B(s) \|_{H^3}^2}\mathrm{d}s
  + \int_{0}^{t}{\| (u(s), B(s)) \|_{W^{1, \infty}}}\mathrm{d}s
  \leqslant C,
\end{equation*}
for all $ 0 < t < T $.
\end{prop}

\begin{proof}
The proof is split into the following eight steps.

\noindent{\bf\textit{Step 1. $ L^{\infty}([0, t]; L^2(\mathbb{R}^3))$ estimate of $ (u, \,B) $.}}

Multiplying the first and second equations in \eqref{Hall-MHD} by $ u $ and $ B $, respectively, integrating over $\mathbb{R}^3$
and adding up, we arrive at
\begin{equation}    \label{Ineq-Prop3.1-1}
\begin{split}
   & \frac{1}{2}\frac{\mathrm{d}}{\mathrm{d}t}
   \| (u(\cdot, \,t), \,  B(\cdot, \,t)) \|_{L^2}^2
   + \| (\nabla u(\cdot, \,t), \,\nabla B(\cdot, \,t)) \|_{L^2}^2 \\
 = & \int{(B \cdot \nabla B) \cdot u}\,\mathrm{d}x
   + \int{(B \cdot \nabla u) \cdot B}\,\mathrm{d}x
   + \int{\operatorname{curl}(\operatorname{curl}B \times B) \cdot B}\,\mathrm{d}x \\
= & \int{(B \cdot \nabla B) \cdot u}\,\mathrm{d}x
   - \int{(B \cdot \nabla B) \cdot u}\,\mathrm{d}x
   + \int{(\operatorname{curl}B \times B) \cdot \operatorname{curl}B}\,\mathrm{d}x \\
= & 0.
\end{split}
\end{equation}

Then, integrating \eqref{Ineq-Prop3.1-1} over $ [0, \,t] $ with respect to $ t $ yields
\begin{equation}
     \begin{split}
   \| (u(\cdot, \,t), \,B(\cdot, \,t)) \|_{L^2}^2
 + \int_0^t \|(\nabla u(\cdot, \,s), \nabla B(\cdot, \,s))\|_{L^2}^2 \, \mathrm{d}s
 \leqslant \| (u_0, \,B_0) \|_{L^2}^2.
     \end{split}
\end{equation}

\noindent{\bf\textit{Step 2. $  L^{\infty}([0, t]; L^2(\mathbb{R}^3))$ estimate
of $ \Omega$ .}}

For any $ p \in [2, \,\alpha] \cap [2, \,\infty) $ with
$ \alpha  > 6 $,
multiplying the first equation in \eqref{Hall-MHD-3} by $ | \Pi |^{p - 2} \Pi $,
we have
\begin{equation}   \label{Eq:Pf-Prop3.2-1}
\begin{split}
  & \frac{1}{p}\frac{\mathrm{d}}{\mathrm{d}t}\| \Pi \|_{L^{p}}^{p}
  + (p - 1)\int{ |\Pi |^{p - 2}\, | \nabla \Pi |^{2}}\mathrm{d}x \\
 = &\,\frac{2}{p}\int{\frac{1}{r}\partial_{r}|\Pi|^{p}}\mathrm{d}x
   + \frac{2}{p + 1}\int{\partial_{z}|\Pi|^{p + 1}}\mathrm{d}x \\
 = &\,\frac{4 \pi}{p}\int_{-\infty}^{\infty}\left(\int_{0}^{\infty}{\partial_{r}|\Pi|^{p}}
     \mathrm{d}r \right)\mathrm{d}z
   + \frac{4 \pi}{p + 1}\int_{0}^{\infty}r\left(\int_{-\infty}^{\infty}{\partial_{z}|\Pi|^{p + 1}}\mathrm{d}z \right)\mathrm{d}r\\
 \leqslant & \, 0,
\end{split}
\end{equation}
here we have used the fact that
$ \displaystyle\lim_{|z| \to \infty}\Pi(r, \,z, \,t) = 0 $ and
$ \displaystyle\lim_{r \to \infty}\Pi(r, \,z, \,t) = 0 $.
Integrating \eqref{Eq:Pf-Prop3.2-1} over $ [0, \,t] $ with respect to $ t $ gives
\begin{equation} \label{p2-1}
   \|\Pi(t)\|_{L^p}
  \leqslant \|\Pi_0\|_{L^p}
  \leqslant \|\Pi_0\|_{L^2}^{\frac{2 (\alpha - p)}{(\alpha - 2) p}}
            \|\Pi_0\|_{L^\alpha}^{\frac{\alpha(p - 2)}{(\alpha - 2) p}},
      \qquad \forall\, p \in [2, \,\alpha] \cap [2, \,\infty).
\end{equation}
In the above process, there are two special cases:

\textbf {Case 1.} $ \alpha = \infty $,
letting $ p \to \infty $ in \eqref{p2-1}, we obtain
\begin{equation*}
   \|\Pi(t)\|_{L^\infty}
  \leqslant   \|\Pi_0\|_{L^\infty},
\end{equation*}
which together with \eqref{p2-1} yields that \eqref{p2-1}
holds for all $ p \in [2, \,\alpha] $.

\textbf {Case 2.} $ p = 2 $, then \eqref{Eq:Pf-Prop3.2-1} can be rewritten as
\[
    \frac{1}{2}\frac{\mathrm{d}}{\mathrm{d}t}\| \Pi(t) \|_{L^{2}}^{2}
  + \| \nabla \Pi(t) \|_{L^{2}}^{2}
  \leqslant 0,
\]
which implies
\begin{equation}  \label{p2-1-2}
    \| \Pi(t)\|_{L^2}^{2} + \int_{0}^{t}\| \nabla \Pi(t) \|_{L^{2}}^{2}\mathrm{d}t \leqslant \|\Pi_0\|_{L^2}^2.
\end{equation}

On the other hand, taking $ L^2-$inner product of the second
equation in \eqref{Hall-MHD-3} with $ \Omega $ yields
\begin{align*}
  \frac{1}{2}\frac{\mathrm{d}}{\mathrm{d}t}\|\Omega\|_{L^2}^2
  + \| \nabla \Omega \|_{L^2}^2
& = -\int (u\cdot \nabla \Omega)\cdot \Omega \,\mathrm{d}x
 +\int \frac{1}{r}\partial_r | \Omega |^2 \,\mathrm{d}x
 -\int\Omega\partial_z \Pi^2 \, \mathrm{d}x \\
& = \frac{1}{2}\int{| \Omega |^{2} \operatorname{div}u}\,\mathrm{d}x
   -2 \pi \int_{-\infty}^{\infty}{| \Omega(0, z, t) |^{2}}\mathrm{d}z
   - \int{ \Omega \partial_{z}\Pi ^{2} }\,\mathrm{d}x \\
& \leqslant - \int{ \Omega \partial_{z}\Pi ^{2} }\,\mathrm{d}x,
\end{align*}
here we have used the fact that $ \operatorname{div}u = 0 $
and $ \displaystyle\lim_{r \to \infty}\Omega(r, \,z, \,t) = 0 $.
A directly calculating gives
\[
  - \int{ \Omega \partial_{z}\Pi ^{2} }\,\mathrm{d}x
 = -2\pi \int_{0}^{\infty}r\left(\int_{-\infty}^{\infty}
   \Omega \partial_{z}\Pi ^{2} \mathrm{d}z\right)\mathrm{d}r
 = 2 \pi \int_{0}^{\infty}r\left(\int_{-\infty}^{\infty}
    \Pi ^{2}\partial_{z}\Omega \mathrm{d}z\right)\mathrm{d}r.
\]
Therefore
\begin{equation*}
\begin{split}
 \frac{1}{2}\frac{\mathrm{d}}{\mathrm{d}t}\|\Omega\|_{L^2}^2
  + \| \nabla \Omega \|_{L^2}^2& = \int{ \Pi ^{2} \partial_{z} \Omega  }\,\mathrm{d}x
 \leqslant \| \Pi \|_{L^{4}}^{2} \| \nabla \Omega \|_{L^2},
     \end{split}
\end{equation*}
which together with Young's inequality and \eqref{p2-1} yields
\begin{equation}\label{p2-3}
  \frac{\mathrm{d}}{\mathrm{d}t} \|\Omega\|_{L^2}^2 + \|\nabla \Omega\|_{L^2}^2
  \leqslant \|\Pi_0\|_{L^4}^4
  \leqslant \|\Pi_0\|_{L^2}^\frac{2\alpha - 8}{\alpha - 2}
            \|\Pi_0\|_{L^\alpha}^\frac{2\alpha}{\alpha - 2}.
\end{equation}
Consequently, integrating \eqref{p2-3} over $ [0, \,t] $ with respect to $ t $ implies
\begin{equation}\label{p2-2}
   \|\Omega(t)\|_{L^2}^2+\int_0^t\|\nabla \Omega(s)\|_{L^2}^2 \, ds
   \lesssim \|\Omega_0\|_{L^2}^2
         + t \|\Pi_0\|_{L^2}^\frac{2\alpha - 8}{\alpha - 2}
            \|\Pi_0\|_{L^\alpha}^\frac{2\alpha}{\alpha - 2}.
\end{equation}
Furthermore,
it is easy to deduce from H\"{o}lder's inequality, Lemma \ref{lem-BS} and \eqref{p2-2} that
\begin{equation}  \label{Ineq:velocity}
   \int_0^t\|\frac{u^r}{r}(s)\|_{L^\infty} \,\mathrm{d}s
\leqslant \sup_{0\leqslant s\leqslant t}\|\Omega(\cdot, \,s)\|_{L^2}^{\frac{1}{2}}
   \int_0^t\|\nabla \Omega(s)\|_{L^2}^{\frac{1}{2}} \,\mathrm{d}s
\lesssim t^\frac{3}{4} + t^\frac{5}{4}.
\end{equation}

\noindent{\bf\textit{Step 3. $  L^{\infty}([0, t]; L^2(\mathbb{R}^3))$ estimate
of $ \nabla u $.}}

For any $ p \in [2, \,\infty) $, multiplying the third equation in \eqref{MHD-2}
by $|B^\theta|^{p-2}B^\theta $, thanks to integrating by
parts and H\"{o}lder's inequality, we have
\begin{equation*}
   \frac{1}{p}\frac{\mathrm{d}}{\mathrm{d}t}\|B^\theta\|_{L^p}^p
  + (p - 1)\!\int {| B^{\theta} |^{ p - 2}
     \,| \nabla B^{\theta}|^{2} }\,\mathrm{d}x
  +  \| \frac{| B^{\theta} |^{\frac{p}{2}}}{r} \|_{L^{2}}^{2}
 = \int |B^\theta|^p|\frac{u^r}{r}|\, \mathrm{d}x
 \leqslant\|B^\theta\|_{L^p}^p\|\frac{u^r}{r}\|_{L^\infty},
\end{equation*}
which together with  Gronwall's inequality and \eqref{Ineq:velocity} yields
\begin{equation*}
   \|B^\theta\|_{L^p}^{p} \leqslant
       \|B^\theta_0\|_{L^p}^{p}
  \exp\left(\int_{0}^{t} \|\frac{u^r}{r}(s)\|_{L^\infty} \, \mathrm{d}s\right)
   \lesssim \|B^\theta_0\|_{L^p}^{p}\exp(
  t^{\frac{3}{4}} + t^{\frac{5}{4}} ).
\end{equation*}
Thus
\begin{equation}\label{p3-1}
   \|B^\theta\|_{L^p}
   \lesssim \|B^\theta_0\|_{L^p}\exp(
  t^{\frac{3}{4}} + t^{\frac{5}{4}} ),
\qquad \forall\,  2 \leqslant p < \infty.
\end{equation}
Combining the inequality $ \|B^\theta_0\|_{L^p} \leqslant \|B^\theta_0\|_{L^2}^{\frac{2}{p}}
\|B^\theta_0\|_{L^\infty}^{1 - \frac{2}{p}} $,
we know that \eqref{p3-1} holds for all
$ p \in [2, \,\infty] $ by passing $ p \to \infty $ in \eqref{p3-1}.

Recall that in cylindrical coordinates the vorticity of the swirl-free axisymmetric velocity $ \omega $ is given in \eqref{doubleCurl}, i.e.,
\[
\omega = \operatorname{curl}u = \omega^\theta  e^\theta
\]
which satisfies
\begin{equation}   \label{Eq:omega}
   \partial_t \omega^\theta +u\cdot\nabla \omega^\theta-(\Delta-\frac{1}{r^2})\omega^\theta-\frac{u^r}{r}\omega^\theta
   =-\partial_z\frac{(B^\theta)^2}{r}.
\end{equation}

Taking the $ L^2 $ inner product of \eqref{Eq:omega} with $ \omega^\theta $ and then using
the incompressible condition $\operatorname{div}u = 0 $ leads
\begin{equation}\label{p4-1}
   \begin{split}
   \frac{1}{2}\frac{\mathrm{d}}{\mathrm{d}t}\|\omega^\theta\|_{L^2}^2 +\|\nabla\omega^\theta\|_{L^2}^2+\|\frac{\omega^\theta}{r}\|_{L^2}^2
   &\leqslant \int \frac{u^r}{r}|\omega^\theta|^2 \, \mathrm{d}x
     - \int \partial_z\frac{(B^\theta)^2}{r}\omega^\theta \,\mathrm{d}x\\
   &:= I_1 + I_2.
   \end{split}
\end{equation}
For $I_1$, we have
\begin{equation}\label{p4-2}
     \begin{split}
   |I_1|\leqslant \|\frac{u^r}{r}\|_{L^\infty}\|\omega^\theta\|_{L^2}^2.
     \end{split}
\end{equation}
For $I_2$, it follows from integration by parts that
\begin{equation}\label{p4-3}
     \begin{split}
   |I_2| =\left|\int \frac{(B^\theta)^2}{r} \partial_z \omega^\theta \,\mathrm{d}x \right|
   & \leqslant \|B^\theta\|_{L^\infty}\|\frac{B^\theta}{r}\|_{L^2}\|\partial_z \omega^\theta\|_{L^2}\\
   &\leqslant \frac{1}{2} \|B^\theta\|_{L^\infty}^2\|\Pi\|_{L^2}^2
    + \frac{1}{2}\|\partial_z \omega^\theta\|_{L^2}^2.
     \end{split}
\end{equation}
\noindent

Substituting \eqref{p4-2} and \eqref{p4-3} into \eqref{p4-1} and using \eqref{p3-1},
we obtain
\begin{equation*}
     \begin{split}
   \frac{\mathrm{d}}{\mathrm{d}t}\|\omega^\theta\|_{L^2}^2 +\|\nabla\omega^\theta\|_{L^2}^2+\|\frac{\omega^\theta}{r}\|_{L^2}^2
   &\lesssim \|\frac{u^r}{r}\|_{L^\infty}\|\omega^\theta\|_{L^2}^2 + \|B^\theta\|_{L^\infty}^2\|\Pi\|_{L^2}^2\\
   &\lesssim \|\frac{u^r}{r}\|_{L^\infty}\|\omega^\theta\|_{L^2}^2 + \|\Pi_{0}\|_{L^2}^2
  \|B^\theta_0\|_{H^2}^2 \exp( t^{\frac{3}{4}} + t^{\frac{5}{4}}),
     \end{split}
\end{equation*}
where we have used the Sobolev imbedding $H^s(\mathbb{R}^3)\hookrightarrow L^\infty(\mathbb{R}^3)$ for $s>\frac{3}{2}$.
And then, we deduce from the Gronwall inequality and \eqref{Ineq:velocity} that
\begin{equation*}
     \begin{split}
  &  \|\omega^\theta(t)\|_{L^2}^2
   + \int_0^t \|\nabla\omega^\theta(s)\|_{L^2}^2 \, \mathrm{d}s +
    \int_0^t \|\frac{\omega^\theta}{r}(s)\|_{L^2}^2 \, \mathrm{d}s\\
   \lesssim\,&\, \left( \|\omega_0^\theta\|_{L^2}^2
   + \|\Pi_{0}\|_{L^2}^2\|B^\theta_0\|_{H^2}^2
  \int_0^t \exp(s^\frac{3}{4} +  s^\frac{5}{4})\, \mathrm{d}s   \right)
   \exp \left(\int_0^t \|\frac{u^r}{r}(s)\|_{L^\infty} \, \mathrm{d}s \right)\\
   \lesssim\,&\, (1+t)\exp(t^{\frac{3}{4}} + t^{\frac{5}{4}}).
     \end{split}
\end{equation*}

In view of
\begin{equation*}
     \begin{split}
   \|\omega\|_{L^2}= \|\omega^\theta\|_{L^2},\quad \mbox{and} \quad  \|\nabla \omega\|_{L^2}^2= \|\nabla \omega^\theta\|_{L^2}^2 + \|\frac{\omega^\theta}{r}\|_{L^2}^2,
     \end{split}
\end{equation*}
we arrive at
\begin{equation}  \label{Ineq:VorEst}
   \|\omega(t)\|_{L^2}^2 +\int_0^t \|\nabla\omega(\tau)\|_{L^2}^2 \, d\tau
  \lesssim (1+t)\exp( t^{\frac{3}{4}} + t^{\frac{5}{4}}),
\end{equation}
which together with Lemma \ref{Lem:grad-curl} yields
\begin{equation}\label{c2-1}
         \|\nabla u(\cdot, \,t)\|_{L^2}^2
       + \int_0^t \|\nabla^2 u(\cdot, \,s)\|_{L^2}^2 \, \mathrm{d}s
 \lesssim (1+t) \exp( t^{\frac{3}{4}} + t^{\frac{5}{4}}).
\end{equation}

\noindent
Moreover, Lemma \ref{lem-BS} and \eqref{Ineq:VorEst} lead to
\begin{equation}    \label{Ineq:L2Linf}
     \begin{split}
   \int_0^t \|u(s)\|_{L^\infty}^2 \, \mathrm{d}s
   &\lesssim  \int_0^t \|\omega^\theta(s)\|_{L^2}\|\nabla\omega^\theta(s)\|_{L^2} \,
     \mathrm{d}s\\
   &\lesssim  \sup_{0\leqslant s\leqslant t}\|\omega^\theta(s)\|_{L^2}
   \left(\int_0^t \|\nabla\omega^\theta(s)\|_{L^2}^2 \, \mathrm{d}s\right)^{\frac{1}{2}}
    \left(\int_0^t 1 \, \mathrm{d}s \right)^{\frac{1}{2}}\\
   &\lesssim (1+t)^{\frac{3}{2}}\exp( t^{\frac{3}{4}} + t^{\frac{5}{4}}).
     \end{split}
\end{equation}

\noindent
\textbf{\textit{Step 4.  $ L^{\infty}([0, \,t]; L^{\infty}(\mathbb{R}^3)) $ estimate for $ u $.}}

For any $ \sigma \in (3, \,\alpha] \cap (3, \,\infty) $, taking the $ L^2 $ inner product of \eqref{Eq:omega}
with $ |\omega^\theta|^{\sigma - 2} \omega^\theta $ yields
\begin{equation*}
     \begin{split}
&   \frac{1}{\sigma}\frac{\mathrm{d}}{\mathrm{d}t}\|\omega^\theta\|_{L^{\sigma}}^{\sigma}
 +\frac{4(\sigma - 1)}{\sigma^2}\|\nabla(|\omega^\theta|^{\frac{\sigma}{2}})\|_{L^2}^2
 +\|\frac{|\omega^\theta|^{\frac{\sigma}{2}}}{r}\|_{L^2}^2 \\
\leqslant
& \,\int \frac{u^r}{r}|\omega^\theta|^{\sigma} \,\mathrm{d}x
 - \int |\omega^\theta|^{\sigma - 2} \omega^\theta\partial_z\frac{(B^\theta)^2}{r}\,\mathrm{d}x\\
\leqslant
& \;\|\frac{u^r}{r}\|_{L^\infty}\|\omega^\theta\|_{L^\sigma}^{\sigma}
 + \int \frac{(B^\theta)^2}{r} \partial_z(|\omega^\theta|^{\sigma - 2} \omega^\theta )\,\mathrm{d}x.
     \end{split}
\end{equation*}
Due to the fact that
\begin{equation}   \label{InnerOmega}
    \partial_z(|\omega^\theta|^{\sigma - 2} \omega^\theta )
  = | \omega^\theta |^{\sigma - 2} \partial_z\omega^\theta
    + \omega^\theta\,\partial_z(|\omega^\theta|^{\sigma - 2})
 = \frac{2(\sigma - 1)}{\sigma}| \omega^\theta |^{\frac{\sigma}{2} - 1}
   \operatorname{sign}(\omega^{\theta})
   \partial_z\big(|\omega^\theta|^{\frac{\sigma}{2}}\big),
\end{equation}
we have
\begin{align}   \label{Esti-Sec}
     \int \frac{(B^\theta)^2}{r} \partial_z(|\omega^\theta|^{\sigma - 2} \omega^\theta )\,\mathrm{d}x  \nonumber
   & \leqslant
  \frac{2(\sigma - 1)}{\sigma} \int | B^\theta |\, \frac{| B^\theta |}{r} \,
     | \omega^\theta |^{\frac{\sigma}{2} - 1}  \,
   | \nabla \big(|\omega^\theta|^{\frac{\sigma}{2}}\big) |\,\mathrm{d}x \\
     & \leqslant \frac{2(\sigma - 1)}{\sigma}
      \|B^\theta\|_{L^\infty}\| \Pi \|_{L^\sigma}
     \|\nabla (|\omega^\theta |^{\frac{\sigma}{2}}) \|_{L^2}
     \|\omega^\theta\|_{L^\sigma}^{\frac{\sigma}{2} - 1}\\ \nonumber
     &\leqslant
  \frac{2(\sigma - 1)}{\sigma^2}\|\nabla (|\omega^\theta |^{\frac{\sigma}{2}}) \|_{L^2}^{2}
   +  C_{\sigma} \|B^\theta\|_{L^\infty}^{\sigma}
      \| \Pi \|_{L^\sigma}^{\sigma}
   + \|\omega^\theta\|_{L^\sigma}^{\sigma},
\end{align}
where $ C_{\sigma} > 0 $ is a constant depending only on $ \sigma $.

Inserting \eqref{Esti-Sec} into \eqref{InnerOmega}, we obtain
\begin{equation*}
     \begin{split}
   \frac{\mathrm{d}}{\mathrm{d}t}\|\omega^\theta\|_{L^\sigma}^{\sigma}
   \leqslant \big(1+ \|\frac{u^r}{r}\|_{L^\infty}\big)\|\omega^\theta\|_{L^\sigma}^\sigma
   + C_{\sigma} \|B^\theta\|_{L^\infty}^{\sigma}
      \| \Pi \|_{L^\sigma}^{\sigma},
\quad \forall\,\sigma \in (3, \,\alpha] \cap (3, \,\infty).
     \end{split}
\end{equation*}
The Gronwall inequality, \eqref{p2-1}\eqref{Ineq:velocity} and \eqref{p3-1} imply that
\begin{equation}   \label{Estim:Omega-sigma}
\begin{split}
   \|\omega^\theta(t)\|_{L^\sigma}
& \lesssim
  \left( \|\omega_{0} \|_{L^\sigma} +
  C_{\sigma}\int_{0}^{t}{\|B^\theta\|_{L^\infty}
      \| \Pi \|_{L^\sigma}}\,\mathrm{d}s \right)
  \exp\left\{ \int_{0}^{t}{\big(1+ \|\frac{u^r}{r}\|_{L^\infty}\big)}\,\mathrm{d}s \right\}\\
 & \lesssim (1+t)\exp( t^{\frac{3}{4}} + t^{\frac{5}{4}}),
\qquad \forall\,\sigma \in (3, \,\alpha] \cap (3, \,\infty).
\end{split}
\end{equation}
In particular, if $ \alpha = \infty $,
letting $ \sigma \to \infty $ in \eqref{Estim:Omega-sigma},
we can also get that  \eqref{Estim:Omega-sigma} holds
for all $ \sigma \in (3, \,\alpha] $.

Then we deduce from \eqref{Estim:Omega-sigma} and
Lemma \ref{Lem:grad-curl} that
\[
   \nabla u \in L^\infty([0, \,t]; \,L^{\sigma}(\mathbb{R}^3)).
\]
Note that $ \sigma > 3 $, by the Sobolev embedding
$ W^{1, \,\sigma}(\mathbb{R}^3) \hookrightarrow C_{0}(\mathbb{R}^3)  $, the space of continuous functions tending to $ 0 $ at infity
(see Theorem 6.7 of \cite{LT2007}), we know
\[
    u \in L^\infty([0, \,t]; \,L^{\infty}(\mathbb{R}^3)) .
\]
\bigbreak

\noindent{\bf\textit{Step 5. $L^1([0, t]; \operatorname{Lip}(\mathbb{R}^3))$ estimate for $u$.}}

Rewriting the equation for vorticity $ \omega =\operatorname{curl} u $ as follows
\begin{equation}\label{p5-1}
     \begin{split}
   \partial_t \omega -\Delta\omega=-\operatorname{curl}(u\cdot \nabla u) - \partial_z(\Pi B^\theta e_\theta).
     \end{split}
\end{equation}

Let $q\in \mathbb{N}$ and $\omega_q :=\Delta_q \omega$. Then localizing in frequency to the vorticity equation \eqref{p5-1} and applying Duhamel formula, we know
\begin{equation*}
     \begin{split}
   \omega_q = e^{t\Delta}\omega_q(0)
- \int_0^t e^{{(t-\tau)}\Delta}\Delta_q \left(\operatorname{curl}(u\cdot \nabla u)\right)(s) \, \mathrm{d}s
 - \int_0^t e^{{(t-s)}\Delta}\Delta_q \left(\partial_z(\Pi B^\theta e_\theta)\right)(s) \, \mathrm{d}s.
     \end{split}
\end{equation*}
By the estimate, see \cite{Che1998},
\begin{equation*}
     \begin{split}
   \|e^{t\Delta}\Delta_q f\|_{L^m}\leqslant Ce^{-ct2^{2q}} \|\Delta_q f\|_{L^m}, \quad \forall\; 1 \leqslant m \leqslant \infty,
     \end{split}
\end{equation*}
and the Bernstein inequality, we get
\begin{equation*}
     \begin{split}
   \|\omega_q\|_{L^p}\lesssim e^{-ct 2^{2q}}\|\omega_q(0)\|_{L^p}
   &+2^{2q} \int_0^t e^{-c(t-\tau)2^{2q}}\|\Delta_q (u\otimes u)(s)\|_{L^p} \,\mathrm{d}s\\
   &+2^q \int_0^t e^{-c(t-\tau)2^{2q}} \|\Delta_q(\Pi B^\theta)(s)\|_{L^p} \, \mathrm{d}s.
     \end{split}
\end{equation*}
Then integrating in time and using convolution inequalities yield
\begin{equation*}
     \begin{split}
   \int_0^t \|\omega_q(s)\|_{L^p} \, \mathrm{d}s\lesssim 2^{-2q}\|\omega_q(0)\|_{L^p}
   +\int_0^t \|\Delta_q (u\otimes u)(s)\|_{L^p}  \, \mathrm{d}s
   +2^{-q} \int_0^t \|\Delta_q(\Pi B^\theta)(\tau)\|_{L^p} \, \mathrm{d}s,
     \end{split}
\end{equation*}
which implies that
\begin{equation*}
     \begin{split}
   \int_0^t \|\omega(s)\|_{B^{\frac{3}{p}}_{p, 1}} \, \mathrm{d}s
   &\lesssim \int_0^t \|\Delta_{-1}\omega(s)\|_{L^p} \, \mathrm{d}s
   +\|\omega_0\|_{B^{\frac{3}{p}-2}_{p, 1}}\\
   &\quad +\int_0^t \|(u\otimes u)(s)\|_{B^{\frac{3}{p}}_{p, 1}}  \, \mathrm{d}s
   +\int_0^t \|(\Pi B^\theta)(s)\|_{B^{\frac{3}{p}-1}_{p, 1}} \, \mathrm{d}s.
     \end{split}
\end{equation*}

Taking $ 3 < p \leqslant 6 $,
applying the Bernstein inequality and  \eqref{Ineq:VorEst} to the first term of the right hand side leads to
\begin{equation*}
   \int_0^t \|\Delta_{-1}\omega(s)\|_{L^p} \, \mathrm{d}s
 \lesssim t \|\omega\|_{L^\infty([0, t]; L^2(\mathbb{R}^3))}
\lesssim (1+t)^\frac{3}{2}\exp (t^\frac{3}{4}+t^\frac{5}{4}).
\end{equation*}
For the second term of the right hand side,  Besov embedding implies
\begin{equation*}
     \begin{split}
   \|\omega_0\|_{B^{\frac{3}{p}-2}_{p, 1}}\lesssim \|u_0\|_{B^{\frac{3}{p}-1}_{p, 1}}
   \lesssim \|u_0\|_{B^{\frac{1}{2}}_{2, 1}} \lesssim \|u_0\|_{H^1}.
     \end{split}
\end{equation*}
Applying Besov embedding, law products and interpolation inequality, we have
\begin{equation*}
     \begin{split}
    \|u\otimes u\|_{B^{\frac{3}{p}}_{p, 1}}
   &\lesssim \|u\otimes u\|_{B^{\frac{3}{2}}_{2, 1}}
    \lesssim \|u\|_{L^\infty}\|u\|_{B^{\frac{3}{2}}_{2, 1}}\\
    &\lesssim \|u\|_{L^\infty}\|u\|_{L^2}+ \|u\|_{L^\infty}\|\nabla u\|_{B^{\frac{1}{2}}_{2, 1}}\\
    &\lesssim \|u\|_{L^\infty}\|u\|_{L^2}+ \|u\|_{L^\infty}\|\nabla u\|_{L^2}^{\frac{1}{2}}\|\nabla^2 u\|_{L^2}^{\frac{1}{2}},
     \end{split}
\end{equation*}
which together with \eqref{c2-1} implies
\begin{align*}
    \|u \otimes u\|_{L^1([0, t], B^{\frac{3}{p}}_{p, 1}(\mathbb{R}^3))}
    \lesssim &\, t^\frac{1}{2}\|u\|_{L^2([0, t], L^\infty(\mathbb{R}^3))}\|u\|_{L^\infty([0, t], L^2(\mathbb{R}^3))}\\
    &  + \|u\|_{L^\frac{4}{3}([0, t], L^\infty(\mathbb{R}^3))}\|\nabla u\|_{L^\infty([0, t], L^2(\mathbb{R}^3))}^{\frac{1}{2}}\|\nabla^2 u\|_{L^2([0, t], L^2(\mathbb{R}^3))}^{\frac{1}{2}}\\
    \lesssim &\, t^\frac{1}{2}\|u\|_{L^2([0, t], L^\infty(\mathbb{R}^3))}\|u\|_{L^\infty([0, t], L^2(\mathbb{R}^3))}\\
    &  + t^{\frac{1}{4}}\|u\|_{L^2([0, t], L^\infty(\mathbb{R}^3))}\|\nabla u\|_{L^\infty([0, t], L^2(\mathbb{R}^3))}^{\frac{1}{2}}\|\nabla^2 u\|_{L^2([0, t], L^2(\mathbb{R}^3))}^{\frac{1}{2}}\\
    \lesssim &\, (1+t)^{\frac{5}{4}}\exp(t^{\frac{3}{4}} + t^{\frac{5}{4}}).
\end{align*}
We use the embedding $L^p\hookrightarrow B^{\frac{3}{p}-1}_{p, 1}$, for $p >3$,
\begin{equation*}
     \begin{split}
   \|\Pi B^\theta\|_{B^{\frac{3}{p}-1}_{p, 1}} \lesssim \|\Pi B^\theta\|_{L^p}\lesssim \|B^\theta\|_{L^\infty}\|\Pi\|_{L^p},
     \end{split}
\end{equation*}
which gives for $ 3 < p \leqslant 6 $
\begin{equation*}
     \begin{split}
   \int_0^t \|\Pi B^\theta(s)\|_{B^{\frac{3}{p}-1}_{p, 1}} \, \mathrm{d}s
 \lesssim \|\Pi_0\|_{L^p}\int_0^t \| B^\theta(s) \|_{ L^\infty}\, \mathrm{d}s
   \lesssim t\exp (t^\frac{3}{4}+t^\frac{5}{4}).
     \end{split}
\end{equation*}

Hence, we have
\begin{equation*}
     \begin{split}
  \int_0^t \|\omega(s)\|_{B^{\frac{3}{p}}_{p, 1}} \, \mathrm{d}s
  \lesssim (1+t)^{\frac{3}{2}}\exp (t^\frac{3}{4}+t^\frac{5}{4}).
     \end{split}
\end{equation*}
and then the Besov embedding $B^{\frac{3}{p}+1}_{p, 1}\hookrightarrow W^{1, \infty}$ implies
\begin{equation}  \label{Lip-u}
     \begin{split}
  \int_0^t \|\nabla u(s)\|_{L^\infty}\, \mathrm{d}s
  \lesssim \int_0^t \|u(s)\|_{B^{\frac{3}{p}+1}_{p, 1}} \, \mathrm{d}s
  \lesssim \int_0^t\|\omega(s)\|_{B^{\frac{3}{p}}_{p, 1}}\, \mathrm{d}s
 \lesssim (1+t)^{\frac{3}{2}}\exp (t^\frac{3}{4}+t^\frac{5}{4}).
     \end{split}
\end{equation}

\noindent{\bf\textit{Step 6. $  L^{2}([0, t]; \,\operatorname{Lip}(\mathbb{R}^3))$ estimate
of $ B $.}}

We first rewrite the second equation in \eqref{Hall-MHD} as
\begin{equation}\label{p7-1}
  \partial_t B + u \cdot \nabla B = \Delta B + B \cdot \nabla u
+ 2\frac{B^{\theta}}{r}\partial_{z}B^{\theta}.
\end{equation}

Taking $ L^{2}-$inner product of \eqref{p7-1} with $ -\Delta B $, and then the H\"older
inequality gives
\begin{equation*}
\begin{split}
    \frac{1}{2}\frac{\mathrm{d}}{\mathrm{d}t}\|\nabla B\|_{L^2}^2
   + \| \Delta B \|_{L^2}^2
& = \int{(u \cdot \nabla B - B \cdot \nabla u
    - 2 \frac{B^{\theta}}{r}\partial_{z}B^{\theta}) \cdot \Delta B}
   \,\mathrm{d}x \\
& \leqslant \big( \| u \|_{L^{\infty}} \| \nabla B \|_{L^{2}}
      + \| \nabla u \|_{L^{\infty}} \| B \|_{L^{2}}  +
      2 \| \frac{B^{\theta}}{r} \|_{L^{6}} \| \partial_{z}B^{\theta} \|_{L^{3}}\big) \| \Delta B \|_{L^{2}}\\
& \leqslant \big( \| u \|_{L^{\infty}} \| \nabla B \|_{L^{2}}
      + \| \nabla u \|_{L^{\infty}} \| B \|_{L^{2}}  +
      2 \| \frac{B^{\theta}}{r} \|_{L^{6}} \| \nabla B \|_{L^{3}}\big) \| \Delta B \|_{L^{2}},
\end{split}
\end{equation*}
In view of $ \| \nabla B \|_{L^{3}} \leqslant \| \nabla B \|_{L^{2}}^{\frac{1}{2}}\| \nabla B \|_{L^{6}}^{\frac{1}{2}} $ and $ H^1(\mathbb{R}^3)\hookrightarrow L^6(\mathbb{R}^3) $,
we obtain
\begin{equation*}
\begin{split}
&\frac{1}{2}\frac{\mathrm{d}}{\mathrm{d}t}\|\nabla B\|_{L^2}^2
   + \| \Delta B \|_{L^2}^2   \\
 \lesssim \,&\,  \| u \|_{L^{\infty}} \| \nabla B \|_{L^{2}} \| \Delta B \|_{L^{2}}
      + \| \nabla u \|_{L^{\infty}} \| B \|_{L^{2}}  \| \Delta B \|_{L^{2}}
      + \| \frac{B^{\theta}}{r} \|_{L^{6}} \| \nabla B \|_{L^{2}}^{\frac{1}{2}}
        \| \Delta B \|_{L^{2}}^{\frac{3}{2}}\\
 \lesssim  \,&\, \frac{1}{2} \| \Delta B \|_{L^{2}}^{2}
      + \| u \|_{L^{\infty}}^{2} \| \nabla B \|_{L^{2}}^{2}
      + \| \nabla u \|_{L^{\infty}}^{2} \| B \|_{L^{2}}^{2}  +
        \| \frac{B^{\theta}}{r} \|_{L^{6}}^{4} \| \nabla B \|_{L^{2}}^{2}.
\end{split}
\end{equation*}
Thus,
\begin{equation*}
     \begin{split}
  \frac{\mathrm{d}}{\mathrm{d}t}\|\nabla B\|_{L^2}^{2}
  + \| \Delta B \|_{L^2}^2
 \lesssim \left(\| u\|_{L^\infty}^{2} + \|\frac{B}{r}\|_{L^6}^4 \right)
   \|\nabla B\|_{L^2}^{2} + \| B \|_{L^2}^2   \|\nabla u\|_{L^\infty}^{2}.
     \end{split}
\end{equation*}
Thanks to Gronwall's inequality, we obtain
\begin{equation}    \label{H1-Est-B}
     \begin{split}
 & \|\nabla B(t)\|_{L^2}^{2} + \int_{0}^{t}{\|\Delta B(s)\|_{L^2}^{2}}\mathrm{d}s \\
 \lesssim \,& \,
    \left( \|\nabla B_0\|_{L^2}^{2} +
   \int_0^t  \| B(s) \|_{L^2}^{2}\| \nabla u(s) \|_{L^\infty}^{2} \mathrm{d}s  \right)
   \exp\left( \int_0^t \big(\| u(s) \|_{L^\infty} + \|\frac{B}{r}(s)\|_{L^6}^{4}\big)\mathrm{d}s\right) \\
\lesssim \,&\, \big(C_{0} + C_{0}(1 + t)^{\frac{3}{2}}
\exp(t^{\frac{3}{4}} + t^{\frac{5}{4}}) \big)
   \exp\big( (1 + t)^{\frac{3}{2}}
\exp(t^{\frac{3}{4}} + t^{\frac{5}{4}}) \big) \\
\lesssim \,&\, \exp\left((1+t)^{\frac{7}{4}}
\exp(t^{\frac{3}{4}} + t^{\frac{5}{4}}) \right),
     \end{split}
\end{equation}
where we have used
\eqref{p2-1-2}\eqref{p2-2}\eqref{Ineq:velocity}\eqref{Lip-u} and the Sobolev embedding
$ H^1(\mathbb{R}^3)\hookrightarrow L^6(\mathbb{R}^3) $.
Therefore
\begin{equation} \label{H2Est-B}
   \| B \|_{L^{\infty}([0,\; t]; \,H^{1}(\mathbb{R}^{3}))}
 + \| B \|_{L^{2}([0,\; t]; \,H^{2}(\mathbb{R}^{3}))}
  \lesssim  \| (u_{0}, \,B_{0}) \|_{L^{2}}
    + \exp\left((1+t)^{\frac{7}{4}} \exp t^{\frac{5}{4}} \right).
\end{equation}

Set
\begin{equation*}
   \gamma = \gamma(\alpha) = 
\begin{cases}
\, \frac{6 \alpha}{6 + \alpha}, 
    & \quad \alpha \in (6, \,\infty),\\
\, 6, & \quad \alpha = \infty. 
\end{cases}
\end{equation*}
Then we can easily find that $ \gamma > 3 $, 
\begin{align*}
   \| B \cdot \nabla u \|_{L^2([0,\,t];\, L^{\gamma}(\mathbb{R}^3))} 
& \leqslant 
  \| B \|_{L^\infty([0,\,t];\, L^{\alpha}(\mathbb{R}^3))} 
  \| \nabla u \|_{L^2([0,\,t];\, L^{6}(\mathbb{R}^3))} \\
& \lesssim 
  \| B \|_{L^\infty([0,\,t];\, L^{2}(\mathbb{R}^3))}^
      {\frac{2}{\alpha}}
  \| B \|_{L^\infty([0,\,t];\, L^{\infty}(\mathbb{R}^3))}^
     {1 - \frac{2}{\alpha}} 
  \| \nabla^{2} u \|_{L^2([0,\,t];\, L^{2}(\mathbb{R}^3))} \\
& \lesssim (1 + t)\exp(t^{\frac{3}{4}} + t^{\frac{5}{4}}),
\end{align*}
\begin{align*}
   \| u \cdot \nabla B \|_{L^2([0,\,t];\, L^{\gamma}(\mathbb{R}^3))} 
& \leqslant 
  \| u \|_{L^\infty([0,\,t];\, L^{\alpha}(\mathbb{R}^3))} 
  \| \nabla B \|_{L^2([0,\,t];\, L^{6}(\mathbb{R}^3))} \\
& \lesssim 
 \| u \|_{L^\infty([0,\,t];\, L^{2}(\mathbb{R}^3))}^{\frac{2}{\alpha}} 
 \| u \|_{L^\infty([0,\,t];\, L^{\infty}(\mathbb{R}^3))}^{1 - \frac{2}{\alpha}} 
\| \Delta B \|_{L^2([0,\,t];\, L^{2}(\mathbb{R}^3))} \\
& \lesssim 
\exp\left((1+t)^{\frac{7}{4}}
\exp(t^{\frac{3}{4}} + t^{\frac{5}{4}}) \right),
\end{align*}
and
\begin{align*}
   \| \frac{B^{\theta}}{r} \partial_{z}B^{\theta} \|_{L^2([0,\,t];\, L^{\gamma}(\mathbb{R}^3))} 
& \leqslant 
  \| \frac{B^{\theta}}{r} \|_{L^\infty([0,\,t];\, L^{\alpha}(\mathbb{R}^3))} 
  \| \nabla B^{\theta} \|_{L^2([0,\,t];\, L^{6}(\mathbb{R}^3))}
   \\
& \lesssim 
  \| \Pi \|_{L^\infty([0,\,t];\, L^{\alpha}(\mathbb{R}^3))} 
  \| \Delta B \|_{L^2([0,\,t];\, L^{2}(\mathbb{R}^3))} \\
& \lesssim 
 \exp\left((1+t)^{\frac{7}{4}}
\exp(t^{\frac{3}{4}} + t^{\frac{5}{4}}) \right).
\end{align*}
Therefore
\begin{equation*}
    \partial_{t}B - \Delta B = B \cdot \nabla u - u \cdot \nabla B
    + 2 \frac{B^{\theta}}{r}\partial_{z}B^{\theta}
\in L^{2}([0, \,t]; \,L^{\gamma}(\mathbb{R}^{3})).
\end{equation*}
By the regularity theory of the heat equation, we deduce that
\begin{equation*}
   \Delta B \in L^{2}([0, \,t]; \,L^{\gamma}(\mathbb{R}^{3})).
\end{equation*}

Since $ \gamma > 3 $
for all $ \alpha \in (6,\,\infty] $, 
by the Sobolev embedding
$ W^{1, \,\gamma}(\mathbb{R}^3) \hookrightarrow C_{0}(\mathbb{R}^3)  $, we know that
\begin{equation}  \label{nablaB}
  \nabla B \in L^{2}([0, t]; L^{\infty}(\mathbb{R}^{3})).
\end{equation}

\noindent{\bf\textit{Step 7. $  L^{1}([0, t]; L^{2}(\mathbb{R}^3)) $ estimate
of $ \nabla^{3} u $.}}

Note that the equation of vorticity
\begin{equation*}
     \begin{split}
   \partial_t \omega -\Delta\omega=-\operatorname{curl}(u\cdot \nabla u) - \partial_z(\Pi B^\theta ),
     \end{split}
\end{equation*}
we obtain from the regularity theory of the heat equation that
\begin{equation*}
     \begin{split}
   \|\Delta \omega\|_{L^1([0, t]; L^2(\mathbb{R}^3))}
   &\,
\lesssim \| \operatorname{curl}(u\cdot \nabla u)\|_{L^1([0, t]; L^2(\mathbb{R}^3))} + \|\partial_z(\Pi B^\theta )\|_{L^1([0, t]; L^2(\mathbb{R}^3))}.
     \end{split}
\end{equation*}
It deduces from Lemma \ref{lem-Besov} that
\begin{equation*}
     \begin{split}
 \| \operatorname{curl}(u\cdot \nabla u)\|_{L^2(\mathbb{R}^3))}\leqslant \|u\cdot \nabla u\|_{H^1}
\lesssim \|u\|_{L^\infty}\|\nabla u\|_{H^1} +\|u\|_{H^1}\|\nabla u\|_{L^\infty},
     \end{split}
\end{equation*}
and
\begin{equation*}
     \begin{split}
 \|\partial_z(\Pi B^\theta )\|_{L^2(\mathbb{R}^3)}
& \leqslant \|B^\theta \partial_z \Pi \|_{L^2(\mathbb{R}^3)} +\|\Pi \partial_z B^\theta \|_{L^2(\mathbb{R}^3)}\\
 &\leqslant \|B^\theta \|_{L^\infty(\mathbb{R}^3)}\|\nabla \Pi \|_{L^2(\mathbb{R}^3)}+ \|\Pi\|_{L^2(\mathbb{R}^3)}\|\nabla B^\theta \|_{L^\infty(\mathbb{R}^3)}.
     \end{split}
\end{equation*}
Thus, by \eqref{p2-1-2}\eqref{p3-1}\eqref{c2-1}\eqref{Ineq:L2Linf} and \eqref{nablaB}, we have
\begin{align*}
    &\|\Delta \omega\|_{L^1([0, t]; L^2(\mathbb{R}^3))}\\
     \lesssim &\,  \| \operatorname{curl}(u\cdot \nabla u)\|_{L^1([0, t]; L^2(\mathbb{R}^3))} + \|\partial_z(\Pi B^\theta )\|_{L^1([0, t]; L^2(\mathbb{R}^3))}\\
     \lesssim &\,  \|u\|_{L^2([0, t]; L^\infty(\mathbb{R}^3))}\|\nabla u\|_{L^2([0, t]; H^1(\mathbb{R}^3))} +
   \|u\|_{L^\infty([0, t]; H^1(\mathbb{R}^3))}\|\nabla u\|_{L^1([0, t]; L^\infty(\mathbb{R}^3))}\\
   &\quad + t^\frac{1}{2}\,\|B^\theta\|_{L^\infty([0, t]; L^\infty(\mathbb{R}^3))}\|\nabla \Pi\|_{L^2([0, t]; L^2(\mathbb{R}^3))}
   + t^\frac{1}{2} \,\|\Pi_0\|_{L^2}\|\nabla B^\theta \|_{L^2([0, t]; L^\infty(\mathbb{R}^3))}\\
\lesssim  &\,  \exp\left((1+t)^2 \exp( t^{\frac{3}{4}} + t^{\frac{5}{4}}) \right).
\end{align*}
Combining \eqref{Ineq:VorEst}, we get
\begin{align*}
  \| \omega \|_{L^1([0, t]; H^2(\mathbb{R}^3))}
& \lesssim
   \| \omega\|_{L^1([0, t]; L^2(\mathbb{R}^3))}
 + \| \nabla \omega\|_{L^1([0, t]; L^2(\mathbb{R}^3))}
 + \|\Delta \omega\|_{L^1([0, t]; L^2(\mathbb{R}^3))} \\
& \lesssim
  \exp\left((1+t)^2 \exp( t^{\frac{3}{4}} + t^{\frac{5}{4}}) \right),
\end{align*}
which gives
\begin{equation}   \label{Est-grad3u}
   \int_0^t \|\nabla^3 u(s)\|_{L^2}\, \mathrm{d} s
  \lesssim \int_0^t \|\omega(s)\|_{H^2}\, \mathrm{d}s
  \lesssim \exp\left((1+t)^{2} \exp( t^{\frac{3}{4}} + t^{\frac{5}{4}}) \right).
\end{equation}

\noindent{\bf\textit{Step 8. $  L^{\infty}([0, t]; H^{2}(\mathbb{R}^3))$ estimate
of $ B $.}}

Applying the operator $\nabla^2$ to the second equation in \eqref{Hall-MHD} leads to
\begin{equation*}
    \begin{split}
    \partial_t \nabla^2 \!B + u\cdot \nabla\nabla^2 \!B - B \cdot \nabla\nabla^2 u
    +\nabla^{2}\!\operatorname{curl}(\operatorname{curl}B \times B) =
     [\nabla^2, B\!\cdot\!\nabla]u -[\nabla^2, u\!\cdot\!\nabla]B + \Delta\nabla^{2}\!B.
   \end{split}
\end{equation*}

Taking the $L^2$ inner product of the above equation
with $ \nabla^2 B $, we obtain
\begin{equation*}
     \begin{split}
   \frac{1}{2}\frac{\mathrm{d}}{\mathrm{d}t} \|\nabla^2 B\|_{L^2}^2
       + \|\nabla^{3} B\|_{L^2}^2
   &=\int B\cdot \nabla\nabla^2 u : \nabla^2 B\, \mathrm{d}x
       -\int[\nabla^2, u\cdot\nabla]B : \nabla^2 B\, \mathrm{d}x\\
   &\quad  + \int[\nabla^2, B\cdot\nabla]u : \nabla^2 B\, \mathrm{d}x
       - \int{\nabla^{2}\!\operatorname{curl}(\operatorname{curl}B \times B) : \nabla^2 B}
       \mathrm{d}x\\
   &:=\sum_{i=1}^{4} J_i.
     \end{split}
\end{equation*}
Next we estimate $ J_i $ term by term.
For $J_1$, we deduce from H\"{o}lder's inequality and Young's inequality that
\begin{align*}
     |J_1|&\leqslant \|B\|_{L^\infty}\|\nabla^3 u\|_{L^2}\|\nabla^2 B\|_{L^2}
     \lesssim \| B \|_{L^\infty}^2 \|\nabla^3 u\|_{L^2}
     + \|\nabla^3 u\|_{L^2}\|\nabla^2 B\|_{L^2}^2,
\end{align*}

Thanks to the interpolation estimate $\|f\|_{L^3}\lesssim \|f\|_{L^2}^{\frac{1}{2}}\|\nabla f\|_{L^2}^{\frac{1}{2}}$, Lemma \ref{lem:Commun}, and Young's inequality, we get
\begin{equation*}
     \begin{split}
     |J_2|&\leqslant \|[\nabla^2, u\cdot\nabla]B\|_{L^2}\|\nabla^2 B\|_{L^2}\\
     &\lesssim \left(\|\nabla u\|_{L^\infty}\|\nabla^2 B\|_{L^2}+ \|\nabla^2 u\|_{L^3}\|\nabla B\|_{L^6}\right)
     \|\nabla^2 B\|_{L^2}\\
     &\lesssim \|\nabla u\|_{L^\infty}\|\nabla^2 B\|_{L^2}^2
     + \|\nabla^2 u\|_{L^2}^{\frac{1}{2}}\|\nabla^3 u\|_{L^2}^{\frac{1}{2}}\|\nabla^2 B\|_{L^2}^{2}\\
     &\lesssim \Big(\|\nabla u\|_{L^\infty}
   + \|\nabla^2 u\|_{L^2} + \|\nabla^3 u\|_{L^2}\Big)\|\nabla^2 B\|_{L^2}^2.
     \end{split}
\end{equation*}
Similarly,
\begin{equation*}
     \begin{split}
     |J_3|&\leqslant \|[\nabla^2, B\cdot\nabla]u\|_{L^2}\|\nabla^2 B\|_{L^2}\\
     &\lesssim \left(\|\nabla B\|_{L^6}\|\nabla^2 u\|_{L^3}+ \|\nabla u\|_{L^\infty}\|\nabla^2 B\|_{L^2}\right)
     \|\nabla^2 B\|_{L^2}\\
     &\lesssim \Big(\|\nabla u\|_{L^\infty}
   + \|\nabla^2 u\|_{L^2} + \|\nabla^3 u\|_{L^2}\Big)\|\nabla^2 B\|_{L^2}^2.
     \end{split}
\end{equation*}
For $ J_{4} $, we have
\begin{align*}
   J_{4}  &=  - \sum_{i, j = 1}^{3}\int{
  \partial_{i}\partial_{j}\!\operatorname{curl}(\operatorname{curl}B \times B) \cdot
    \partial_{i}\partial_{j}B}\,\mathrm{d}x   \\
&    = - \sum_{i, j = 1}^{3}   \int{
  \partial_{i}\partial_{j}\!(\operatorname{curl}B \times B) \cdot
    \partial_{i}\partial_{j} \operatorname{curl}B}\,\mathrm{d}x   \\
&    = - \sum_{i, j = 1}^{3} \int{ \left(
  \partial_{i}\partial_{i} (\operatorname{curl}B \times B)
    - \partial_{i}\partial_{i}\!\operatorname{curl}B \times B \right)
    \cdot \partial_{j}\partial_{j}\!\operatorname{curl}B}\,\mathrm{d}x,
\end{align*}
i.e.,
\begin{align*}
  J_{4} &  =   -\int{ \left(
  \Delta (\operatorname{curl}B \times B)
    - \Delta\operatorname{curl}B \times B \right)
    : \Delta\operatorname{curl}B}\,\mathrm{d}x \\
  & =  \int{ \left(
  \Delta (B \times \operatorname{curl}B)
    - B \times (\Delta\operatorname{curl}B ) \right)
    : \Delta\operatorname{curl}B}\,\mathrm{d}x \\
  & \leqslant
   \| \Delta (B \times \operatorname{curl}B )
    - B \times \Delta\operatorname{curl}B  \|_{L^{2}}  \| \Delta\operatorname{curl}B \|_{L^{2}}.
\end{align*}
Thus, it deduce from Lemma \ref{lem:Commun} that
\begin{align*}
   | J_{4} |
& \lesssim  (\| \nabla B \|_{L^{\infty}} \| \nabla \operatorname{curl}B  \|_{L^{2}}
    + \| \nabla^{2} B \|_{L^{2}}
   \| \operatorname{curl}B  \|_{L^{\infty}})  \| \nabla^{3}B \|_{L^{2}}  \\
&  \lesssim  \| \nabla B \|_{L^{\infty}} \| \nabla^{2}B \|_{L^{2}}
    \| \nabla^{3}B \|_{L^{2}} \\
& \lesssim  \frac{1}{2}\| \nabla B \|_{L^{\infty}}^{2}  \| \nabla^{2} B \|_{L^{2}}^{2}
       + \frac{1}{2} \| \nabla^{3}B \|_{L^2}^{2}
\end{align*}

Collecting all the above estimates together, we arrive at
\begin{equation*}
     \begin{split}
  \frac{\mathrm{d}}{\mathrm{d}t} \|\nabla^2 B\|_{L^2}^2 + \| \nabla^{3}B \|_{L^2}^{2}
\lesssim \,&  \left(\|\nabla u\|_{L^\infty} +\|\nabla^{2} u\|_{L^2}
   + \|\nabla^{3} u\|_{L^2} + \|\nabla B\|_{L^\infty}^2   \right)\|\nabla^2 B\|_{L^2}^2 \\
   & \, + \| B \|_{L^\infty}^2  \|\nabla^3 u\|_{L^2},
     \end{split}
\end{equation*}
which together with
\eqref{c2-1}\eqref{Lip-u}\eqref{H1-Est-B}\eqref{H2Est-B}\eqref{Est-grad3u}
and Gronwall's inequality yields
\begin{equation*}
     \begin{split}
  &\|\nabla^2 B(t)\|_{L^2}^2 + \int_{0}^{t}{\|\nabla^3 B(s)\|_{L^2}^2}\mathrm{d}s\\
   \lesssim &\, \exp \left\{ \!\int_0^t \Big(\|\nabla u(s)\|_{L^\infty}
    + \|\nabla^{2} u(s)\|_{L^2}
  + \|\nabla^{3} u(s)\|_{L^2} + \|\nabla B(s)\|_{L^\infty}^2 \Big) \,
     \mathrm{d}s \right\}\\
  &\, \times \left( \|B_0\|_{H^2}^2
   + \| B \|_{L^{\infty}([0,\,t];\; L^\infty(\mathbb{R}^{3}))}^2\int_0^t {  \|\nabla^3 u\|_{L^2}}\, \mathrm{d}s  \right)\\
 \leqslant &\, C(t, \,\| u_{0} \|_{H^{1}}, \,\| B_{0} \|_{H^{2}}).
     \end{split}
\end{equation*}
\end{proof}

\begin{proof}[Proof of Theorem \ref{thm-main}]
With the Proposition \ref{prop-1}, by taking advantage of the local existence and uniqueness result, i.e., Lemma \ref{lem-0}, we complete the proof of Theorem \ref{thm-main}.
\end{proof}

\bigbreak

\noindent {\bf Acknowledgments.}The authors would like to thank professor Guilong Gui for his valuable
comments and suggestions. The work is partially supported by the National Natural Science Foundation of China under the grants 11571279, 11601423 and 11931013.

\end{document}